\documentclass[11pt]{article}
\usepackage{amssymb,amsmath,amsthm,enumerate}
\usepackage{amsmath,epsfig,amssymb,amsbsy,verbatim,array,color,graphics,graphicx,bbm,pgfplots}
\pgfplotsset{compat=1.14}

\newcount\mwhr\newcount\mwmin
\mwhr=\number\time\divide
\mwhr by60 \mwmin=\mwhr\multiply\mwmin by-60
\advance\mwmin by\number\time{}
\newtheorem{theorem}{Theorem}

\newtheorem{Thm}{Theorem}
\newtheorem{Lemma}[Thm]{Lemma}

\newtheorem{Coro}[Thm]{Corollary}

\begin{document}

\title{\bf On a Conjecture of Nagy on Extremal Densities}

\author{
A. Nicholas Day\thanks{Institutionen f\"or Matematik och Matematisk Statistik, Ume{\aa} Universitet, 901 87 Ume{\aa}, Sweden.
Email: nicholas.day@umu.se.  Research supported by Swedish Research Council grant 2016-03488.}
\and
Amites Sarkar\thanks{Department of Mathematics, Western Washington University, Bellingham, Washington 98225, USA.}
}

\maketitle

\begin{abstract}
We disprove a conjecture of Nagy on the maximum number of copies $N(G,H)$ of a fixed graph $G$ in a large
graph $H$ with prescribed edge density. Nagy conjectured that for all $G$, the quantity $N(G,H)$ is asymptotically
maximised by either a quasi-star or a quasi-clique.  We show this is false for infinitely many graphs, the smallest of which has $6$ vertices and $6$ edges.  We also propose some new conjectures for the behaviour of $N(G,H)$, and present some evidence for them.
\end{abstract}

\section{Introduction}

Let $G$ be a fixed small graph, let $\beta\in(0,1)$, and let $n$ be a large positive integer. The problem of
asymptotically minimizing the number of copies of $G$ in a large graph $H$ on $n$ vertices, with edge density $\beta$,
is a very well-studied problem in extremal graph theory. It generalises the forbidden subgraph problem, and has received
much attention in recent years. For instance, a famous conjecture of Sidorenko~\cite{Sid} states that when $G$
is bipartite, the minimiser $H$ is quasirandom, and a celebrated theorem of Reiher~\cite{Reiher} solves the
problem for complete graphs (the minimiser is close to a Tur\'an graph).

In this paper we will study the opposite problem: given $G$, how do we {\it maximise} the number of copies of
$G$ in $H$? As before, $H$ will have order $n$ and edge density $\beta$, and, as before, we are mainly interested
in asymptotics: we will write ``maximiser" for ``asymptotic maximiser" throughout. This problem also has a long
history, going back at least to Ahlswede and Katona~\cite{AK}, who studied the case when $G=P_2$, the path
with two edges. (Throughout this paper, $P_l$ denotes a path with $l$ {\it edges}.) Roughly speaking, Ahlswede and
Katona proved that, for $\beta>\frac{1}{2}$, the maximiser is a {\it quasi-clique}, i.e., a clique $K$, with another vertex
joined to a subset of $V(K)$, together with some isolated vertices; we note that the size of the clique is uniquely
determined by $n$ and $\beta$. When $\beta<\frac{1}{2}$, they proved further that the maximiser is instead a {\it quasi-star}:
the complement of a quasi-clique; the parameters of the quasi-star are again uniquely determined by $n$ and $\beta$.
In short, the maximiser is first a quasi-star and then a quasi-clique, with the ``flip" occurring at $\beta=\frac{1}{2}$. The
paper~\cite{AK} in fact contains an exact result for $G=P_2$, which is surprisingly complicated.

On the other hand, for some graphs $G$, the maximiser is always a quasi-clique, regardless of $\beta$.  Alon showed that this is the case for any graph with with $\alpha^*(G)=v/2$, where $\alpha^*(G)$ is the fractional independence number of $G$, and $v$ is the number of vertices of $G$; see Section \ref{section: definitions and notation}
for a definition of the fractional independence number of a graph. A result of Janson, Oleszkiewicz and Ruci\'{n}ski~\cite{JOR} can be used to show that this condition is in fact an ``if and only if" condition, that is, if the maximiser of $G$ is a quasi-clique for all $\beta$, then $\alpha^*(G)=v/2$.  Alon's result (which was originally formulated
in a different way) was generalised to hypergraphs by Friedgut and Kahn~\cite{FK}.

These results leave many questions open. To restate the basic one: given a small graph $G$ on $v$ vertices, we would
like to know which large graphs $H$ asymptotically maximise $N(G,H)$, the number of unlabelled copies of $G$ in $H$,
where $H$ runs over all graphs on $n$ vertices and edge density $\beta$. Interest in this problem was revitalised
by its connection with graphons, and subsequently by the work of Nagy~\cite{N}, who solved it for $G=P_4$.
Nagy's result is that for $P_4$, as for $P_2$, the maximiser is first a quasi-star and then a quasi-clique, with the
flip occurring this time at $\beta=0.0865...$, instead of $\frac{1}{2}$. By contrast, odd-length paths such as $P_3$ are
covered by Alon's theorem: for them, the maximiser is always a quasi-clique.

Note that, up until now, we have been tacitly assuming that the maximiser is in some sense unique; however, this might
conceivably not be the case, so we should strictly speaking write ``an (asymptotic) maximiser", rather than ``the maximiser".

At the end of his paper~\cite{N}, Nagy posed three questions. The first asked for an exact, not just asymptotic, result
for $P_4$. The second asked whether, for every graph $G$, and every edge density $\beta$, the quantity $N(G,H)$ is always (asymptotically)
maximised when $H$ is either a quasi-star or a quasi-clique. The third question was more cautious: given $G$,
is there always a nontrivial threshold $\beta_G<1$, such that the quasi-clique (asymptotically) maximises $N(G,H)$
for $\beta>\beta_G$?

While the first question remains out of reach, the third was recently answered in the affirmative by Gerbner, Nagy,
Patk\'os and Vizer~\cite{GNPV}, using a recent result of Reiher and Wagner~\cite{RW} (who had in turn extended an earlier
result of Kenyon, Radin, Ren and Sadun~\cite{KRRS}). Reiher and Wagner answered Nagy's second question affirmatively for
stars, i.e., they proved that the maximiser is first a quasi-star and then a quasi-clique when $G$ is a star, with the
flip always occurring at $\beta<1$. Both~\cite{KRRS} and~\cite{RW} make extensive use of graphons.

In this paper, we give a negative answer to the second question. Specifically, we exhibit a 6-vertex graph
$G_6$, such that neither the quasi-star nor the quasi-clique asymptotically maximise $N(G_6,H)$ at any edge density
$\beta\in(0,0.016)$. More generally, we show that any graph $G$ on $v$ vertices satisfying
$\alpha^*(G)>\max(\alpha(G),v/2)$ is also a counterexample.

The plan of this paper is as follows. In Section 2 we introduce some basic definitions and notation, as well as the
graphs $T^e_n(q)$, which lie at the heart of the paper. Section 3 contains our main result, which provides a family
of counterexamples to Nagy's conjecture. In Section 4 we describe the prior work of Janson, Oleszkiewicz and Ruci\'{n}ski mentioned above, and discuss how it relates to these counterexamples. Finally, in Section 5 we propose some new conjectures, which we hope will inspire
further progress in this area.

\section{Definitions and notation}\label{section: definitions and notation}

For a fixed small graph $G$ on $v$ vertices, we write
\[
{\rm ex}(n,e,G)=\max\{N(G,H):|H|=n,e(H)\leqslant e\},
\]
where $N(G,H)$ is the number of unlabelled copies of $G$ in $H$. We will also wish to consider labelled copies of $G$;
in this case we will fix a labelling $G_{l}$ of $G$, and work with $N_l(G_l,H)=N(G,H)|{\rm Aut G}|$ instead. For example,
$N(P_4,C_7)=7$ and $N(P_4,K_5)=60$. With $G$ and $\beta\in[0,1]$ fixed, a family of graphs $(H_n)_{n\ge 1}$, such that each $H_{n}$ has $n$ vertices and edge density $\beta+O(1/n)$, is an {\it asymptotic maximiser} for $G$ at edge density $\beta$ if
\[
N(G,H_n)=(1+o(1)){\rm ex}(n,\lfloor\beta n^2/2\rfloor,G).
\]
A graph homomorphism from $G_{l}$ to $H$ is a map $f:V(G_{l})\rightarrow V(H)$ such that $\{f(u),f(w)\}\in E(H)$ for
all $\{u,w\} \in E(G_{l})$.  We write ${\rm hom}(G_{l},H)$ for the number of homomorphisms from $G_{l}$ to $H$.
Given a family of graphs $(H_{n})_{n \geqslant 1}$ such that $|V(H_{n})| = n$, we have that $N_{l}(G,H_{n}) \leqslant
{\rm hom}(G_{l},H_{n})$ and also $N_{l}(G,H_{n}) = \left(1+o(1)\right){\rm hom}(G_{l},H_{n})$ as $n \rightarrow \infty$. As such,
we define
\begin{equation}
t(G_{l},H_{n}) = \lim_{n \rightarrow \infty} \frac{N_{l}(G_{l},H_{n})}{n^{v}}=
\lim_{n \rightarrow \infty} \frac{{\rm hom}(G_{l},H_{n})}{n^{v}}; \nonumber
\end{equation}
the limit will exist for all the families $H_n$ we consider. We will switch between working with ${\rm hom}(G_{l},H_{n})$,
$N_{l}(G,H_{n})$ and $t(G_{l},H_{n})$, depending on which is convenient at the relevant time.

Given a graph $G$, a function $\phi:V(G) \rightarrow [0,1]$ such that $\phi(u) + \phi(w) \leqslant 1$ for all
$\{u,w\} \in E(G)$ is known as a \textit{fractional independence weighting} of $G$.  The \textit{fractional independence number}
of $G$, written $\alpha^{*}(G)$, is defined as the maximum of $\sum_{u \in V(G)} \phi(u)$ over all fractional independence
weightings of $G$. We will make crucial use of the following result of Nemhauser and Trotter~\cite{NemTro}: any graph $G$ has
a maximal weighting (one that realises $\alpha^*(G)$) in which all the weights are either $0, \frac{1}{2}$ or 1. As is customary, we
write $\alpha(G)$ for the usual ($0-1$) independence number of $G$.

We now turn to some specific families of graphs.
Given $n$ and $e\le{n\choose 2}$, there is a unique quasi-clique $K_{n}^{e}$ with $n$ vertices and $e$ edges. To define it,
we first write $e={a\choose 2}+b$, where $0\le b<a$. The graph $K_{n}^{e}$ is a complete graph $K_a$ with $a$ vertices, with an
additional vertex joined to $b$ vertices of $K_a$, and $n-a-1$ isolated vertices. Likewise, there is a unique quasi-star
$S_{n}^{e}$ with $n$ vertices and $e$ edges; this is just the complement of $K_{n}^{e'}$, where $e'={n\choose 2}-e$.

Here we will be interested in asymptotics only. Thus we will replace the number of edges $e$ by the (asymptotic)
{\it edge density} $\beta=2e/n^2$. For asymptotic purposes, $K_{n}^{e}$ is a clique of size $\sqrt{\beta}n$, and
$V(S_{n}^{e})$ can be partitioned into two sets $R_{S}$ (red vertices) and $B_{S}$ (blue vertices), where the red vertices
span a clique of size $(1-\sqrt{1-\beta})n$, the blue vertices form an independent set of size $\sqrt{1-\beta}n$, and every
blue vertex is joined to every red vertex. Here, and in what follows, we omit floor functions for ease of notation; our ``approximate versions" of $K_{n}^{e}, S_{n}^{e}$ and (in the next paragraph) $T_{n}^{e}(q)$ will not have exactly $n$
vertices and $e$ edges. However, we will have, for instance, $e(T_{n}^{e}(q))=\left(1+O\left(1/n\right)\right)e$, and this is more than enough for
our asymptotic estimates.

Let $q \in [0,1]$.  The following graph $T=T_{n}^{e}(q)$, with (asymptotically) $n$ vertices and $e$ edges, will prove useful.
We partition the vertices of $T$ into three sets $Y_{T}$ (yellow), $R_{T}$ (red) and $B_{T}$ (blue), with the following sizes:
\begin{eqnarray}
|Y_{T}| &=& \sqrt{\beta}qn,  \nonumber\\
|R_{T}| &=& \left(1 - \sqrt{1 - \beta\left(1-q^{2}\right)}\right)n,\nonumber \\ 
|B_{T}| &=& \left(\sqrt{1 - \beta\left(1-q^{2}\right)} - \sqrt{\beta}q\right)n.\nonumber
\end{eqnarray}
The sets $Y_{T}$ and $R_{T}$ both span cliques, while $B_{T}$ is an independent set. Also, every vertex in $R_{T}$ is connected
to every vertex in $Y_{T}$ and $B_{T}$. It is easy to check that $T_{n}^{e}(q)$ has the required number of vertices and edges.
Moreover, we have that $T_{n}^{e}(0)  = S_{n}^{e}$ and $T_{n}^{e}(1) = K_{n}^{e}$, so that $T_{n}^{e}(q)$ interpolates between
$S_{n}^{e}$ and $K_{n}^{e}$. See Figure~\ref{Fig1} for a picture of $S_{n}^{e}$ and $T_{n}^{e}(q)$.
\begin{figure}[ht]
\centering
\includegraphics[scale = 0.8]{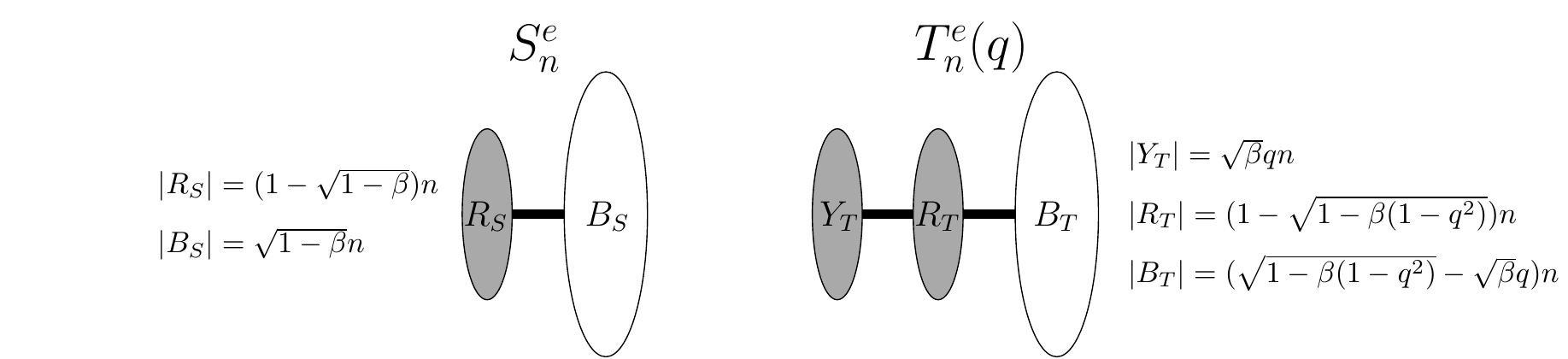}
\caption{The graphs $S_{n}^{e}$ and $T_{n}^{e}(q)$ where $e = \frac{\beta n^2}{2}$.  The shaded sets are cliques,
while the white sets are independent.  A black line between two sets indicates they are fully connected. }
\label{Fig1}
\end{figure}

\section{Main result}
In this section we prove the following theorem, which disproves the conjecture of Nagy.
\begin{theorem}\label{theorem: main theorem}
Let $G_{l}$ be a graph on $v$ vertices such that $\alpha^{*}(G_{l})> \max\left(\alpha\left(G_{l}\right),\frac{v}{2}\right)$.
Fix $q \in (0,1)$. Then there exists $\epsilon = \epsilon(G_{l},q) > 0$ such that, for all $\beta \in (0,\epsilon)$, we have
\begin{equation}
t\left(G_{l},T^{e}_{n}\left(q\right)\right) >\max\left(t\left(G_{l},K^{e}_{n}\right),t\left(G_{l},S^{e}_{n}\right)\right). \nonumber
\end{equation}
\end{theorem}

We remark that there are infinitely many graphs $G$ that have $\alpha^{*}(G)> \max\left(\alpha\left(G\right),\frac{v}{2}\right)$.  For example, for $a \geqslant 3$, $b \geqslant 2$, consider the following graph that has $a+b+1$ vertices.  We start by taking a clique $K_{a}$, we then add a single vertex $u$ to our graph and connect it to one vertex of $K_{a}$.  We finally add $b$ more vertices to our graph, and connect them all to $u$.  It is easy to see that $\alpha^{*}(G) = b + \frac{a}{2}$, $\alpha(G) = b+1$ and $\frac{v}{2} = \frac{a+b+1}{2}$.  The smallest such graph occurs when $a = 3$ and $b = 2$.  We call this graph $G_{6}$, and we study it more carefully in Section \ref{subsection:an explicit counterexample}.

The proof of Theorem \ref{theorem: main theorem} will follow from the two homomorphism counting lemmas below; however we
first need to introduce some more notation.  Given a labelled graph $G_{l}$, let $\Phi(G_{l})$ be the set of fractional
independence weightings of $G_{l}$ in which every vertex receives a weight in $\{0,\frac{1}{2},1\}$.  Given $\phi \in\Phi(G_{l})$, let
\begin{align*}
R_{\phi} &= \Big\{w \in V(G): \phi(w) = 0\Big\}, \nonumber \\
Y_{\phi} &= \Big\{w \in V(G): \phi(w) = \frac{1}{2}\Big\}, \nonumber \\
B_{\phi} &= \Big\{w \in V(G): \phi(w) = 1\Big\}, \nonumber
\end{align*}
and let $r_{\phi} = |R_{\phi}|$, $y_{\phi} = |Y_{\phi}|$, and $b_{\phi} = |B_{\phi}|$.  For all $q \in [0,1]$, we define
\begin{align*}
y(q) &= \sqrt{\beta}q, \\ \nonumber
r(q) &= 1 - \sqrt{1 - \beta(1-q^2)},\\ \nonumber
b(q) &= \sqrt{1-\beta(1-q^2)} - \sqrt{\beta}q.\nonumber
\end{align*}

Recall that for any $q \in [0,1]$, the graph $T^{e}_{n}(q)$ has three vertex classes $Y_T, R_T$ and $B_T$, with
$|Y_{T}| = y(q)n$, $|R_{T}| = r(q)n$, and $|B_{T}| = b(q)n$, where $Y_{T}$ and $R_{T}$ both span cliques, $B_{T}$ is an
independent set, and every vertex in $R_{T}$ is connected to every vertex in $Y_{T}$ and $B_{T}$.

\begin{Lemma}\label{lemma: counting copies}
Let $G_{l}$ be a labelled graph, and fix $\beta,q \in [0,1]$. Then
\begin{equation}
t\left(G_{l},T^{e}_{n}\left(q\right)\right) = \sum_{\phi \in \Phi(G_{l})}y(q)^{y_{\phi}}r(q)^{r_{\phi}}b(q)^{b_{\phi}}. \nonumber
\end{equation}
\end{Lemma}

\begin{Lemma}\label{lemma: vanishing beta}
Let $G_l$ be a labelled graph on $v$ vertices with no isolated vertices.  Fix $q \in (0,1)$, and let $\beta  \rightarrow 0$.
Then there exist constants $C_{1} = C_{1}(G_l,q)>0$ and $C_{2} = C_{2}(G_l)>0$ such that the following all hold:
\begin{enumerate}
\item $t\left(G_{l},K^{e}_{n}\right) = \beta^{\frac{v}{2}}$,
\item $t\left(G_l,T^{e}_{n}\left(q\right)\right) = C_{1}\left( \beta^{v-\alpha^{*}\left(G\right)}
+ O\left(\beta^{v-\alpha^{*}(G) + \frac{1}{2}}\right)\right)$,
\item $t\left(G_l,S^{e}_{n}\right) = C_{2}\left(\beta^{v-\alpha\left(G\right)}
+ O\left(\beta^{v-\alpha\left(G\right) + 1}\right)\right)$.
\end{enumerate}
\end{Lemma}
In Corollary \ref{corollary: constant calculations} below we give explicit values of the constants $C_{1}$ and $C_{2}$.  We remark that, just as with the constants $C_{1}$ and $C_{2}$, the constants hidden in the big $O$ notation in this lemma may depend on $G_{l}$ and $q$, but no other variables.
\begin{proof}[Proof of Lemma \ref{lemma: counting copies}]
Given a homomorphism $f$ from $G_{l}$ to $T_{n}^{e}(q)$, we give a weighting $\phi_f$ to the vertices of $G_{l}$ in the
following way:
\begin{equation}
\phi_{f}(u) = \begin{cases} 0 &\mbox{if } f(u) \in R_{T}, \\
  \frac{1}{2} &\mbox{if } f(u) \in Y_{T}, \\
1 &\mbox{if } f(u) \in B_{T}. \end{cases} \nonumber
\end{equation}
It is easy to see that $\phi_{f}$ is a fractional independence weighting of $G_{l}$.  Given $\phi \in\Phi(G_{l})$, let ${\rm hom}_{\phi}(G_{l},T_{n}^{e}(q))$
be the number of homomorphisms $f$ from $G_{l}$ to $T_{n}^{e}(q)$ such that $\phi_{f} = \phi$, and let
\begin{equation}
t_{\phi}(G_{l},T_{n}^{e}(q)) = \lim_{n \rightarrow \infty} \frac{{\rm hom}_{\phi}(G_{l},T_{n}^{e}(q))}{n^{v}}. \nonumber
\end{equation}
A homomorphism $f$ from $G_{l}$ to $T_{n}^{e}(q)$ has the property that $\phi_{f} = \phi$ if and only if
$f(R_{\phi}) \subseteq R_{T}$, $f(Y_{\phi}) \subseteq Y_{T}$, and $f(B_{\phi}) \subseteq B_{T}$.  Therefore
\begin{equation}\label{equation: weighting formula}
t_{\phi}\left(G_{l},T^{e}_{n}\left(q\right)\right) = y(q)^{y_{\phi}}r(q)^{r_{\phi}}b(q)^{b_{\phi}}.
\end{equation}
As each homomorphism from $G_{l}$ to $T_{n}^{e}(q)$ gives rise to a fractional independence weighting of $G_{l}$ as
described above, we have that
\begin{equation}\label{equation: label count to weighting}
t(G_{l},T_{n}^{e}(q)) = \sum_{\phi \in \Phi(G_{l})}t_{\phi}(G_{l},T_{n}^{e}(q)).
\end{equation}
Combining (\ref{equation: weighting formula}) with (\ref{equation: label count to weighting}) gives the result.
\end{proof}
\begin{proof}[Proof of Lemma \ref{lemma: vanishing beta}]
We start by proving the first part of the lemma.  The quasi-clique $K^{e}_{n}$ consists of a clique $X$ on $\sqrt{\beta}n$
vertices, and $(1-\sqrt{\beta})n$ isolated vertices. As $G_{l}$ has no isolated vertices, a map $f:V(G_{l}) \rightarrow V(K^{e}_{n})$ is a homomorphism from $G$ to $K^{e}_{n}$
if and only if $f\left(V\left(G_{l}\right)\right) \subseteq V(X)$, and so $t\left(G_{l},K^{e}_{n}\right) = \beta^{\frac{v}{2}}$,
as required.

We now proceed by proving the remaining two parts of the lemma.  Recall that the Taylor series for $1-\sqrt{1-x}$ about $0$ is
$\frac{x}{2} + O(x^{2})$.  Thus, as $\beta \rightarrow 0$, we have that
\begin{align*}
r(q) &=\frac{\beta\left(1-q^{2}\right)}{2} + O\left(\beta^{2}\right), \nonumber \\
b(q) &= 1 + O\left(\sqrt{\beta}\right).
\end{align*}
Therefore, combining this with (\ref{equation: weighting formula}) from the proof of Lemma \ref{lemma: counting copies},
we have that
\begin{equation}\label{equation: weighting formula vanishing beta}
t_{\phi}\left(G_{l},T^{e}_{n}\left(q\right)\right) =\left(\frac{1-q^{2}}{2}\right)^{r_{\phi}}q^{y_{\phi}}\left( \beta^{\left(r_{\phi}
+ \frac{y_{\phi}}{2}  \right)} + O \left(\beta^{\left(r_{\phi} + \frac{y_{\phi}}{2} + \frac{1}{2} \right)}\right) \right),
\end{equation}
for all $\phi \in\Phi(G_{l})$.  Suppose first that $q = 0$, so that we are counting homomorphisms into $S_{n}^{e}$.  In order for
(\ref{equation: weighting formula vanishing beta}) to not equal zero, we must have that $y_{\phi} = 0$, which corresponds
precisely to there being no vertex $u \in V(G_{l})$ such that $\phi(u) = \frac{1}{2}$.  In this case, we can rewrite
(\ref{equation: weighting formula vanishing beta}) as
\begin{equation}\label{equation: hom count q=0}
t_{\phi}(G_{l},S^{e}_{n}) = 2^{-r_{\phi}}\left( \beta^{r_{\phi} } + O \left(\beta^{\left(r_{\phi}  + 1 \right)}\right) \right).
\end{equation}
We remark that since $q =0$, we have $b(q) = 1 + O(\beta)$ rather than $b(q) = 1 + O(\sqrt{\beta})$, and so our correction
term in (\ref{equation: hom count q=0}) is an improvement over that in (\ref{equation: weighting formula vanishing beta}).
Among all $\phi \in \Phi(G_{l})$ such that $y_{\phi} = 0$, we have that $r_{\phi} \geqslant v - \alpha(G_{l})$, and equality
occurs if and only if  $\sum_{u \in V(G_{l})} \phi(u) =\alpha(G_{l})$. Let $C'_{2}$ be the number of $\phi \in \Phi(G_{l})$
such that $y_{\phi} = 0$ and $\sum_{u \in V(G_{l})} \phi(u) =\alpha(G_{l})$.  Then, combining (\ref{equation: hom count q=0})
with (\ref{equation: label count to weighting}) from the proof of Lemma \ref{lemma: counting copies}, we have that
\begin{equation}
t(G_{l},S^{e}_{n}) = 2^{\alpha\left(G_{l}\right)-v}C_{2}'\left(\beta^{v-\alpha\left(G_{l}\right)}
+ O\left(\beta^{v-\alpha\left(G_{l}\right) + 1}\right)\right), \nonumber
\end{equation}
completing the proof of the third part of the lemma.

To prove the second part of the lemma, we proceed in a similar fashion.  Fix $q \in (0,1)$.  For all $q$ in this range,
we have that $y(q),r(q),b(q)>0$.  For all $\phi \in \Phi(G_{l})$ we have that $r_{\phi} + \frac{y_{\phi}}{2} \geqslant v
- \alpha^{*}(G_{l})$, and equality occurs if and only if $\sum_{u \in V(G_{l})} \phi(u) =\alpha^{*}(G_{l})$.  Thus for a
suitable constant $C_{1}$, as in the previous case, we have that
\begin{equation}
t\left(G_{l},T^{e}_{n}\left(q\right) \right)= C_{1}\left( \beta^{v - \alpha^{*}(G_{l})  } + O \left(\beta^{v - \alpha^{*}
\left(G_{l}\right)+\frac{1}{2}}\right) \right), \nonumber
\end{equation}
as required.
\end{proof}

\begin{Coro}\label{corollary: constant calculations}
Let $\alpha = \alpha(G_{l})$ and $\alpha^{*} = \alpha^{*}(G_{l})$.  For each $0 \leqslant i \leqslant v - \alpha_{*}$,
let $\tilde{C_{i}}$ be the number of $\phi \in \Phi(G_{l})$ such that $\sum_{u \in V(G_{l})} \phi(u) =\alpha^{*}$ and
$r_{\phi} = i$.  Let $A(G_{l})$ be the number of independent sets $X$ in $G_{l}$ such that $|X| = \alpha$.  Then
\begin{equation}
C_{1} = \sum_{i = 0}^{v - \alpha_{*}} \tilde{C_{i}}\left(\frac{1-q^{2}}{2}\right)^{i}q^{2(v-\alpha^{*}-i)}, \nonumber
\end{equation}
and
\begin{equation}
C_{2} =2^{\alpha-v} A(G_{l}).\nonumber
\end{equation}
\end{Coro}
\begin{proof}
The calculation for $C_{1}$ follows directly from the proof of Lemma \ref{lemma: vanishing beta}, by taking care to
calculate the ``suitable constant" mentioned at the end of the proof.

To calculate $C_{2}$, we first note that if $\phi \in \Phi(G_{l})$ is such that $y_{\phi} = 0$ and $\sum_{u \in V(G_{l})}
\phi(u) =\alpha$, then $B_{\phi}$ is an independent set in $G_{l}$, and $b_{\phi} = \alpha(G)$.  On the other hand, given
an independent set $X \in V(G_{l})$, the weighting $\phi:V(G_{l}) \rightarrow \{0,1\}$, given by $\phi(u) = 1$ if and only if $u \in X$, is a fractional independence weighting of $G_{l}$ with $y_{\phi} = 0$ and $\sum_{u \in V(G_{l})} \phi(u) = |X|$.  Thus $C'_{2}$,
as defined in the proof of the theorem, is equal to $A(G_{l})$, and so the second part of the corollary follows.
\end{proof}

Theorem \ref{theorem: main theorem} follows immediately from Lemma \ref{lemma: vanishing beta}. Indeed, if
$\alpha^{*}(G)> \max\left(\alpha\left(G\right),\frac{v}{2}\right)$, and $q \in (0,1)$, then both $t(G_l,K^e_n)=o(t(G_l,T^e_n(q)))$ and
$t(G_l,S^e_n)=o(t(G_l,T^e_n(q)))$ as $\beta\to0$.

Another consequence of Lemma \ref{lemma: vanishing beta} is that, as $\beta\to0$, we have that $t(G_l,K^e_n)=o(t(G_l,S^e_n))$ if
$\alpha(G)>v/2$ and $t(G_l,S^e_n)=o(t(G_l,K^e_n))$ if $\alpha(G)<v/2$. When $\alpha(G)=v/2$, we may apply the following result,
which was proved independently by many people (see Cutler and Radcliffe~\cite{CR} for references and a short proof).

\begin{theorem}~\cite{CR}
If $G$ is a graph with $n$ vertices, $\alpha(G)\le\alpha$ and $0\le k\le n$, then
\[
i_k(G)\le i_k(K_{n_1}\cup K_{n_2}\cup\cdots\cup K_{n_{\alpha}}),
\]
where $\sum n_i=n,n_1\le n_2\le\cdots\le n_{\alpha}\le n_1+1$, and $i_k(G)$ denotes the number of independent sets of
size $k$ in $G$.
\end{theorem}

\noindent Taking $\alpha=k=v/2$, we see that $i_k(G)\le i_k(K_2\cup K_2\cup\cdots\cup K_2)=2^{v/2}$, so that, in our notation,
$A(G_l)\le2^{v/2}$ when $\alpha(G)=v/2$. Together with Lemma \ref{lemma: vanishing beta} and Corollary
\ref{corollary: constant calculations}, this implies the following.

\begin{theorem} With asymptotic notation as $\beta\to0$,
\[
{\rm max}\{t(G_l,S_n^e),t(G_l,K_n^e)\}\sim t(G_l,K_n^e){\rm \ if\ and\ only\ if\ }\alpha(G)\le v/2.
\]
\end{theorem}

\subsection{An explicit counterexample}\label{subsection:an explicit counterexample}
Throughout this subsection, $G_{l}$ will be the (labelled) graph with $V(G_{l})=[6]$ and
\begin{equation}
E(G_{l}) = \big\{ \{1,2\},\{1,3\},\{2,3\},\{3,4\},\{4,5\},\{4,6\}\big\},\nonumber
\end{equation}
as in Figure \ref{Fig2}, and $T_{n}^{e}$ will be the graph $T_{n}^{e}(1/\sqrt{2})$.  Let $G_6$ be an unlabelled copy of $G_{l}$.
We will show that, for $\beta\in(0,0.016)$, the graph $T_{n}^{e}$ has many more copies of $G_6$ than either $K_{n}^{e}$ or $S_{n}^{e}$.

\begin{figure}[ht]
\centering
\includegraphics[scale = 0.6]{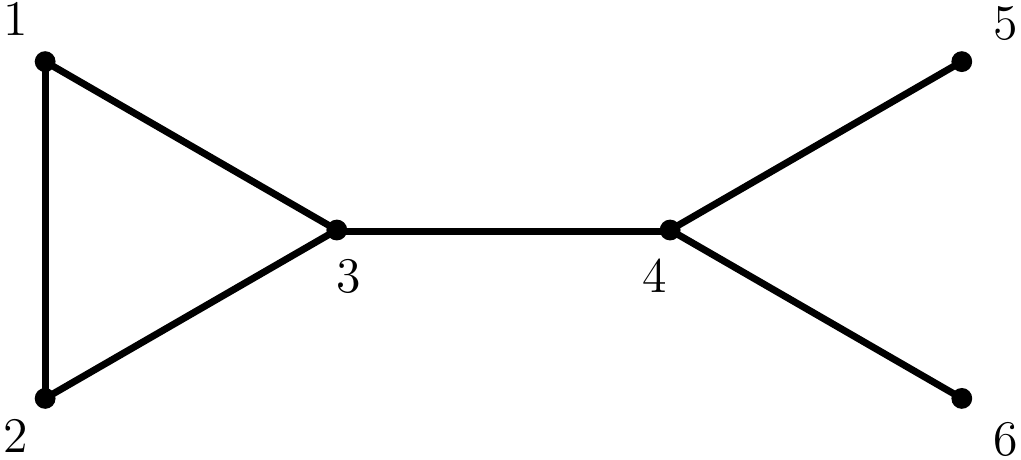}
\caption{The labelled graph $G_{l}$; the unlabelled version is $G_6$.}
\label{Fig2}
\end{figure}

\begin{theorem}
For $\beta \in (0,0.016)$, we have $t(G_{l},T_{n}^{e}) > t(G_{l},K_{n}^{e}) > t(G_{l},S_{n}^{e})$.
\end{theorem}
\begin{proof}
First, as in the proof of Lemma \ref{lemma: vanishing beta}, it is easy to see that $t(G_{l},K_{n}^{e})=\beta^{3}$.  We next turn
to calculating $t(G_{l},S_{n}^{e})$.  Recall that $S_{n}^{e} = T_{n}^{e}(0)$, and that $y(0) = 0, r(0) = 1-\sqrt{1-\beta}$ and
$b(0) = \sqrt{1-\beta}$.  By Lemma \ref{lemma: counting copies}, we have that
\begin{equation}\label{equation: example 1}
t(G_{l},S_{n}^{e}) = \sum_{\phi \in \Phi'(G_{l})}r(0)^{r_{\phi}}b(0)^{b_{\phi}},
\end{equation}
where $\Phi'(G_{l})$ is the set of fractional weightings of $G$ in which every vertex receives weight $0$ or $1$.  Given such a
fractional weighting $\phi$, let $c_{\phi}$ be a colouring of the vertices of $G$ where vertices such that $\phi(u) = 0$ are coloured
red, and vertices with $\phi(u) = 1$ are coloured blue. In Figure \ref{Fig3} below, we classify the elements of $\Phi'(G_{l})$ by
the number of blue vertices in their corresponding colourings.  Note that no such colouring can have more than $\alpha(G_{l}) = 3$
blue vertices.

\begin{figure}[ht]
\centering
\includegraphics[scale = 0.8]{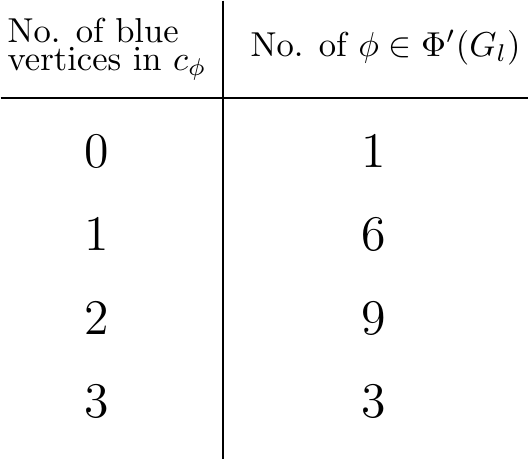}
\caption{The number of $\phi \in \Phi'(G_{l})$ whose colouring has a given number of blue vertices.}
\label{Fig3}
\end{figure}

\noindent From Figure \ref{Fig3} and (\ref{equation: example 1}) we see that
\begin{equation}\label{eq1}
t(G_{l},S_{n}^{e}) = r^{6} + 6r^{5}b + 9r^{4}b^{2} + 3r^{3}b^{3}=r^{3}\left(6r^{2}b + 9rb^{2} + 3b^{3}\right),
\end{equation}
where $r = r(0)$ and $b = b(0)$.

To calculate $t(G_{l},T_{n}^{e})$, we now let
\begin{align*}
y &= y\left(\frac{1}{\sqrt{2}}\right) = \sqrt{\frac{\beta}{2}}, \\
r &= r\left(\frac{1}{\sqrt{2}}\right) = 1-\sqrt{1-\frac{\beta}{2}},\\
b &= b\left(\frac{1}{\sqrt{2}}\right) = \sqrt{1-\frac{\beta}{2}} - \sqrt{\frac{\beta}{2}}.
\end{align*}
In a similar fashion to the above, we need to classify all $\phi \in \Phi(G_{l})$.  Given $\phi \in \Phi(G_{l})$, let $c_{\phi}$
be a colouring of the vertices of $G$ where vertices such that $\phi(u) = \frac{1}{2}$ are coloured yellow, vertices such that
$\phi(u) = 0$ are coloured red, and vertices with $\phi(u) = 1$ are coloured blue.  Again, as above, one can list all $\phi \in
\Phi(G_{l})$ by keeping track of how many red and blue vertices the colouring $c_{\phi}$ has.  We omit the table listing these
colourings, but as before we see that
\begin{eqnarray}\label{eq2}
t(G_{l},T_{n}^{e}) &=& (y + r)^{6} + 2(y + r)^{3}r^{2}b + 2 (y + r)^{2}r^{3}b + 2(y + r)^{4}rb \nonumber \\
&+& 2r^{4}b^{2} + 6(y+r)r^3b^2 + (y + r)^{3}rb ^{2} + 3r^{3}b^{3}.
\end{eqnarray}

Plotting a graph of the functions $t(G_{l},T_{n}^{e}), t(G_{l},K_{n}^{e})$ and  $t(G_{l},S_{n}^{e})$ gives the result; see Figure
\ref{Fig4} for a picture of these three functions for $\beta\in[0,0.016]$. We remind the reader that $e = \frac{\beta n^2}{2}$,
and also that the quantities $r$ and $b$ are different in equations (\ref{eq1}) and (\ref{eq2}). Namely they are $r(0)$ and $b(0)$,
and $r(1/\sqrt2)$ and $b(1/\sqrt2)$ respectively.

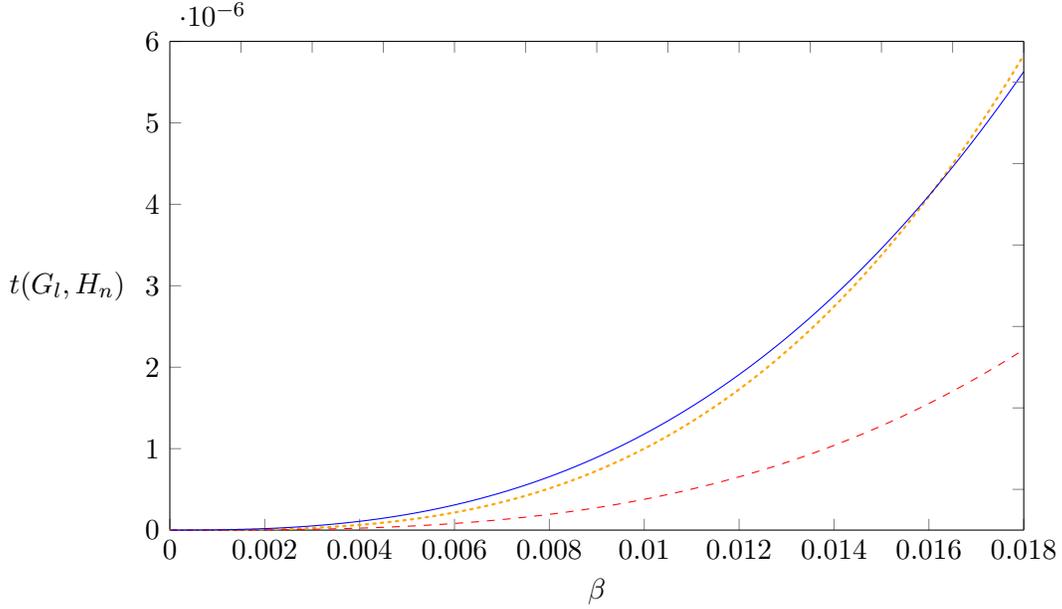
\begin{figure}[ht]
\centering
\definecolor{dyellow}{rgb}{1.0, 0.65, 0.0}
\begin{tikzpicture}
\begin{axis}[
xticklabel style={
/pgf/number format/fixed,
/pgf/number format/precision=5
}, scaled x ticks=false,
ylabel style={rotate=-90},
axis lines* = left,
xmin=0,xmax=0.018,
ymin=0,ymax=0.000006,
xlabel = $\beta$,
ylabel = {$t(G_{l},H_{n})$},
legend pos=north west,
width=0.8\textwidth,
height=0.5\textwidth,
]
\addplot [
domain=0:0.018, 
samples=100, 
color=dyellow,
style=dotted,
line width = 0.3mm,
line cap = round
]
{x^3};
\addplot [
domain=0:0.018, 
samples=100, 
color=red,
style=dashed,
]
{((x/2  + (x^2)/8)^3)*(6*((x/2  + (x^2)/8)^2)*(1-x/2  - (x^2)/8) + 9*(x/2  + (x^2)/8)*((1-x/2  - (x^2)/8)^2) + 3*((1-x/2  - (x^2)/8)^3))};
\addplot [
domain=0:0.018, 
samples=100, 
color=blue,
]
{(((x/2)^(0.5)) + (x/4 + (x^2)/32))^6 + 2*((((x/2)^(0.5)) + (x/4 + (x^2)/32))^3)*((x/4 + (x^2)/32)^2)*(1-x/4 - (x^2)/32 - (x/2)^(0.5)) + 2*((((x/2)^(0.5)) + (x/4 + (x^2)/32))^2)*((x/4 + (x^2)/32)^3)*(1-x/4 - (x^2)/32 - (x/2)^(0.5)) + 2*((((x/2)^(0.5)) + (x/4 + (x^2)/32))^4)*(x/4 + (x^2)/32)*(1-x/4 - (x^2)/32 - (x/2)^(0.5)) + 2*((x/4 + (x^2)/32)^4)*((1-x/4 - (x^2)/32 - (x/2)^(0.5))^2) + 6*(((x/2)^(0.5)) + (x/4 + (x^2)/32))*((x/4 + (x^2)/32)^3)*((1-x/4 - (x^2)/32 - (x/2)^(0.5))^2) + ((((x/2)^(0.5)) + (x/4 + (x^2)/32))^3)*(x/4 + (x^2)/32)*((1-x/4 - (x^2)/32 - (x/2)^(0.5))^2) + 3*((x/4 + (x^2)/32)^3)*((1-x/4 - (x^2)/32 - (x/2)^(0.5))^3)};
\end{axis} 
\begin{axis}[
axis lines* = right,
legend pos=north west,
xticklabels={,,},
yticklabels={,,},
width=0.8\textwidth,
height=0.5\textwidth,
]
\end{axis} 
\end{tikzpicture}
\caption{A graph comparing the functions $t(G_{l},K_{n}^{e})$, $t(G_{l},S_{n}^{e})$ and $t(G_{l},T_{n}^{e})$ on the interval $\beta \in [0,0.018]$.
The blue solid line is the function $t(G_{l},T_{n}^{e})$, the yellow dotted line is the function $t(G_{l},K_{n}^{e})$, and the red dashed line is the function
$t(G_{l},S_{n}^{e})$.  We have that $t(G_{l},K_{n}^{e}) = t(G_{l},T_{n}^{e})$ when $\beta \approx 0.01613474$.}\label{Fig4}
\end{figure}
\end{proof}

Note that we do not claim that $T^{e}_{n}$ is the maximiser for the graph $G_{6}$. We have only shown that there exists $\beta$ such that the maximiser for $G_{6}$ at edge density $\beta$ is neither $K^{e}_{n}$  nor  $S^{e}_{n}$.  Nonetheless, we do believe that, for
all graphs $G$, and for all edge densities $\beta \in [0,1]$, some graph family $H_{n}=T^{e}_{n}(q)$ is the maximiser. We refer the reader to
Section \ref{section: conjectures} for further details on this.

\section{The random connection}

With some effort, a counterexample to Nagy's conjecture can also be read out of some previous results of Janson, Oleszkiewicz and
Ruci\'{n}ski~\cite{JOR}. (We discovered this paper only after we had proved Theorem \ref{theorem: main theorem}.) As part of their
celebrated study on the upper tail for subgraph counts in random graphs, Janson, Oleszkiewicz and Ruci\'{n}ski proved the following
(in our notation).

\begin{theorem}\label{theorem: Janson}
Let $G$ be a graph on $v$ vertices with fractional independence number $\alpha^*(G)$. Then, with $\beta=2e/n^2$,
\[
N(n,e,G)=\Theta(n^v\beta^{v-\alpha^*(G)}).
\]
\end{theorem}

Since it is easy to see that $N(G,K^e_n)=\Theta(n^v\beta^{v/2})$ and $N(G,S^e_n)=\Theta(n^v\beta^{v-\alpha(G)})$, Theorem \ref{theorem: Janson} by
itself shows that neither the quasi-clique nor the quasi-star asymptotically maximises $N(G,H)$, at sufficiently small edge density
$\beta$, if $\alpha^*(G)>\max(\alpha(G),v/2)$. To disprove Nagy's conjecture, one only has to exhibit a single graph satisfying the
last condition (for instance, $G_6$ - there is no such graph on five or fewer vertices).

It is worth describing the lower bound construction in~\cite{JOR}, and its relationship to the graphs $T^e_n(q)$.
(The construction in~\cite{JOR} is expressed in terms of the solution to a linear program; we rephrase it in our notation.)
Given a graph $G$ on $v$ vertices with fractional independence number $\alpha^*(G)$, let $\phi$ be a weighting of $V(G)$
realizing $\alpha^*(G)$. As we've already remarked, we may assume that $\phi$ takes values in $\{0,\frac{1}{2},1\}$. Let $c$
be a sufficiently small constant. For each vertex $u\in V(G)$, we ```blow up" $u$ to an independent set $B_u$ of size
$cn\beta^{1-\phi(u)}$. For each edge $\{u,w\}\in E(G)$, we put a complete bipartite graph (containing $c^2n^2\beta^
{2-\phi(u)-\phi(v)}\le c^2n^2\beta$ edges) between $B_u$ and $B_v$. Call the resulting graph $H$. For sufficiently small
$c$, the graph $H$ has at most $n$ vertices and at most $|E(G)|c^2n^2\beta\le \beta n^2/2$ edges. Moreover, $H$ contains $\prod_u
cn\beta^{1-\phi(u)}=c^vn^v\beta^{v-\alpha^*(G)}$ copies of $G$.

Janson, Oleszkiewicz and Ruci\'{n}ski made no attempt to optimise the constant, and indeed it is not hard to see that one
can improve on their construction by making $H[B_u]$ a clique whenever $\phi(u)\not=1$, and adjusting the sizes of the $B_u$.
In other words, we amalgamate those $B_u$ for which $\phi(u)$ is constant, to get just three sets $B_0,B_1$ and $B_{1/2}$,
``fill in" $B_0$ and $B_{1/2}$ with cliques, and impose the conditions $|H|=n$ and $|E(H)|=e$ while keeping
$|B_i|=\Theta(n\beta^{1-i})$. The result of doing this is just the graph $T^e_n(q)$; the sets $B_0,B_{1/2}$ and $B_1$ are just $R_T,Y_T$
and $B_T$ respectively. Thus, for the lower bound, one only needs to consider a one-parameter family $T^e_n(q)$, instead of
a separate construction for each $G$, and, moreover, this family $T^e_n(q)$ simply consists of graphons with at most three
``steps". We conjecture in the next section that some $T^e_n(q)$ is always asymptotically optimal.

For completeness, we sketch the proof of the upper bound from~\cite{JOR}. To do this, we first re-examine the lower bound construction,
where each vertex $u$ in the small graph $G$ is ``blown up" to an independent set $B_u$ of size $cn^{x_u}$, for some $0\le x_u\le 1$,
and in which the sought-after copies of $G$ are {\it compatible} with the partition $(B_u)_{u\in V(G)}$, i.e., we only look for
copies of $G$ where each $u\in V(G)$ is located in $B_u$. Now, a simple random argument~\cite{FK,JOR} shows that, with
$G$ fixed and $|G|=v$, any large graph $H$ has a vertex partition $V(H)=\bigcup_{u\in V(G)} B_u$ in which at least $v^{-v}N(G,H)$
of the $N(G,H)$ copies of $G$ in $H$ are compatible (in the above sense) with the partition. So it is enough to show,
given a graph $H$, together with a partition of $V(H)$ into $v=|V(G)|$ parts, labelled with the vertices of $G$, that $H$
contains at most $\Theta(n^v\beta^{v-\alpha^*(G)})$ {\it compatible} copies of $G$.

Fixing $G$, and given a partition of $V(H)$, we now aim to choose the edges of $H$ so as to maximise the number $N^c(G,H)$ of
compatible copies of $G$ in $H$. Clearly, if there is no edge from $u$ to $w$ in $G$, we should not put any edges between $B_u$
and $B_w$ in $H$. For edges $\{u,w\}$ of $G$, if we make $E(B_u,B_w)$ a complete bipartite graph, we have exactly the lower bound
construction. The question remains: can we increase $N^c(G,H)$ by increasing the sizes of the parts $B_u$, while
thinning out the edge sets $E(B_u,B_w)$ for $\{u,w\}\in E(G)$? It turns out that the answer is no.

To see this, we again revisit the lower bound construction, in which the parts $B_u$ have sizes $cn^{x_u}$, where
the vertex weights $x_u$ comprise a solution to the following linear program.
\begin{flalign}\label{equation: primal}
&{\rm\bf Maximise\ \ }\sum_u x_u{\rm\bf \ \ \ subject\ to\ \ \ }0\le x_u\le 1{\rm \ and\ \ }uw\in E(G)\Rightarrow x_u+x_v\le 2-\epsilon.
\end{flalign}
Here, $\epsilon=-\log\left(\beta/2\right)/\log n$, so that $\beta/2=n^{-\epsilon}$. Given a weighting $\phi$ of $V(G)$ realizing $\alpha^*(G)$,
a solution to (\ref{equation: primal}) can be obtained by setting
\begin{equation}\label{equation: connection}
x_u=1-\epsilon(1-\phi(u)).
\end{equation}
The dual program is to find nonnegative edge weights $y_{uw}$ and vertex weights $z_u$ of $G$ as below.
\begin{equation}\label{equation: dual}
{\rm\bf Minimise\ \ }\sum_u z_u+(2-\epsilon)\sum_{uw\in E(G)}y_{uw}{\rm\bf \ \ \ subject\ to\ \ \ }u\in V(G)\Rightarrow
z_u+\sum_{uw\in E(G)} y_{uw}\ge 1.
\end{equation}
By linear programming duality, the minimum in (\ref{equation: dual}) is exactly the maximum in (\ref{equation: primal}), and, by
(\ref{equation: connection}), this maximum is just $v-\epsilon(v-\alpha^*(G))$ (yielding the lower bound
$N^c(G,H)\ge Cn^v\beta^{v-\alpha^*(G)}$).

Now each compatible copy of $G$ in $H$ may be considered as a $v$-vertex hyperedge on $V(H)$; together these form
the hypergraph ${\mathcal H}$, whose edges correspond to compatible copies of $G$ in $H$. Rationalizing a solution
to (\ref{equation: dual}) by $a_{uw}=\lceil My_{uw}\rceil$ and $b_{u}=\lceil Mz_u\rceil$, where $M$ is a large positive integer,
we form a sequence of subsets of $V(H)$ by taking each $B_u$ $b_{u}$ times and each $B_u\cup B_w$ $a_{uw}$ times.
By construction, and (\ref{equation: dual}), each vertex in each $B_u$ is covered by at least $M$ of these subsets. For a subset
$V'\subset V(H)$, write ${\rm Tr}({\mathcal H},V')=\{h\cap V':h\in {\mathcal H}\}$ for the trace of ${\mathcal H}$ on $V'$.
Then, by Shearer's lemma~\cite{Shearer},
\begin{align*}
N^c(G,H)=|{\mathcal H}|
&\le\left(\prod_{u\in V(G)}|{\rm Tr}({\mathcal H},B_u)|^{b_u}\prod_{uw\in E(G)}|{\rm Tr}({\mathcal H},B_u\cup B_w)|^{a_{uw}}\right)^{1/M}\\
&\le\left(\prod_{u\in V(G)}n^{b_u}\prod_{uw\in E(G)}n^{(2-\epsilon)a_{uw}}\right)^{1/M}
\to \prod_{u\in V(G)}n^{z_u}\prod_{uw\in E(G)}n^{(2-\epsilon)y_{uw}}\\
&=n^v(\beta/2)^{v-\alpha^*(G)},
\end{align*}
as $M\to\infty$, where, in the last line, we have used the duality theorem of linear programming. 
This is the sought-after upper bound.

We mention for completeness that Janson, Oleszkiewicz and Ruci\'{n}ski's results were generalised to hypergraphs by Dudek, Polcyn and
Ruci\'{n}ski~\cite{DPR}.

\section{Conjectures}\label{section: conjectures}

In this section, we make some new conjectures about the asymptotic value of ${\rm ex}(n,e,G)$. These are essentially the simplest
modifications of Nagy's conjecture which fit the known data.

First we define the concept of an {\it upper profile boundary}. For a fixed labelled graph $G_l$ on $v$ vertices, we look
at the number of homomorphisms ${\rm hom}(G_l,H)$ from $G_l$ to $H$, where $H$ ranges over the set of all unlabelled graphs.
To each graph $H$ with $n$ vertices and $e$ edges, we associate the point
\[
p(G_l,H)=(2e/n^2,{\rm hom}(G_l,H)/n^v)\in[0,1]^2,
\]
whose $x$-coordinate is the edge density of $H$, and whose $y$-coordinate is the {\it homomorphism density} of $G_l$ in $H$.
In this way, each labelled graph $G_l$ gives rise to a {\it profile} $P(G_l)\in[0,1]^2$, defined as the closure of the set
of all the points $p(G_l,H)$. The {\it upper profile boundary} of $G_l$ is the upper boundary of $P(G_l)$;
it is not hard to see that this boundary is the graph of a function $f(G_l,\beta)$ of the edge density $\beta$.
See~\cite{Lo} (page 28) for a picture of the profile of $K_3$.

We return to the graphs $T^e_n(q)$. For a given graph $G_l$, and a given edge density $\beta$ (of $H$), define $f_T(G_l,\beta)$
by the formula
\[
f_T(G_l,\beta)=\sup_{q\in[0,1]}t(G_l,T^e_n(q)),
\]
where $e=\beta n^2/2$, and let $q(G_l,\beta)\in[0,1]$ be the value of $q$ at which the supremum is attained. In other words,
the function $f_T(G_l,\beta)$ is the normalised asymptotic number of copies of $G_l$ in the optimised $T$-graph. Now we are
ready to state our conjectures.

\medskip

\noindent{\bf Conjecture 1.} For all graphs $G_{l}$ and all $\beta \in [0,1]$, we have that $f(G_l,\beta)=f_T(G_l,\beta)$. In other words, for all graphs $G_l$ and all edge densities $\beta$,
some graph family $H_{n} =T^e_n(q)$ asymptotically maximises ${\rm hom}(G_l,H_{n})$ and $N_l(G_l,H_{n})$.

\medskip

\noindent{\bf Conjecture 2.} For each graph $G_l$, we have that $q(G_l,\beta)$ is an increasing function of $\beta$.

\medskip

A slightly stronger version of Conjecture 1 can be most clearly stated in terms of the ``STK notation". Indeed, it appears
that, for each graph $G_l$, there is a partition of the set $[0,1]$ of edge densities into three sets $S,T$ and $K$ (we suppress
the dependence on $G_l$) such that for $\beta\in S$, the quasi-star (asymptotically) maximises $N(G_l,H)$, for $\beta\in T$
some graph $T^e_n(t)$ (with $t\in (0,1)$) maximises $N(G_l,H)$, and for $\beta\in K$, the quasi-clique $K^e_n$ maximises
$N(G_l,H)$. If in addition Conjecture 2 holds, these partitions have a particularly simple form. Indeed, in keeping with the
theorem of Reiher and Wagner~\cite{RW}, only four possibilities can arise:

\medskip

\noindent Type {\bf K}: $K=[0,1]$,

\noindent Type {\bf SK}: $S=[0,\gamma]$ and $K=[\gamma,1]$, for some $\gamma\in(0,1)$,

\noindent Type {\bf TK}: $T=[0,\gamma]$ and $K=[\gamma,1]$, for some $\gamma\in(0,1)$,

\noindent Type {\bf STK}: $S=[0,\gamma],T=[\gamma,\delta]$ and $K=[\delta,1]$, for some $0<\gamma<\delta<1$.

\medskip

With this notation,
Alon~\cite{Alon} characterised graphs of type {\bf K}, Ahlswede and Katona~\cite{AK} proved that $P_2$ is type {\bf SK},
Nagy~\cite{N} proved that $P_4$ is type {\bf SK} and conjectured that all graphs are either type {\bf K} or {\bf SK},
and Reiher and Wagner~\cite{RW} proved that stars are type {\bf SK}, enabling Gerbner, Nagy, Patk\'os and Vizer~\cite{GNPV}
to prove that the type always ends in {\bf--K}. In contrast, the results of Janson, Oleszkiewicz and Ruci\'{n}ski only have a 
bearing on the {\it start} of the type; for instance, the type of $G_6$ cannot begin with either {\bf S--} or {\bf K--}, 
and we conjecture that it is in fact {\bf TK}.  We remark that we are unaware of any graphs of type {\bf STK}, and would be very interested to know whether or not such graphs exist.  See Figure \ref{Fig5} in the Appendix for a summary of the various types of all connected\footnote{Any graph on $5$ or fewer vertices that is not connected must contain either an isolated vertex or an isolated edge, and so can be reduced to a smaller graph.} graphs on at most $5$ vertices.  

We conclude our paper with a weaker version of Conjecture 1, concerning the the behaviour of $f$ and $f_T$ as $\beta\to 0$:

\medskip 

\noindent{\bf Conjecture 3.} As $\beta\to 0$, we have that $f(G_l,\beta)\sim f_T(G_l,\beta)$ for all graphs $G_{l}$.

\medskip

\section{Acknowledgements}
The work of both authors was supported by STINT foundation grant IB2017-7360.  Research of the first author was funded by a grant from the Swedish Research Council, whose support is also gratefully
acknowledged.

\begin{figure}[ht]
\textbf{\Large{Appendix}}\par\medskip\medskip
\centering
\includegraphics[scale = 0.67]{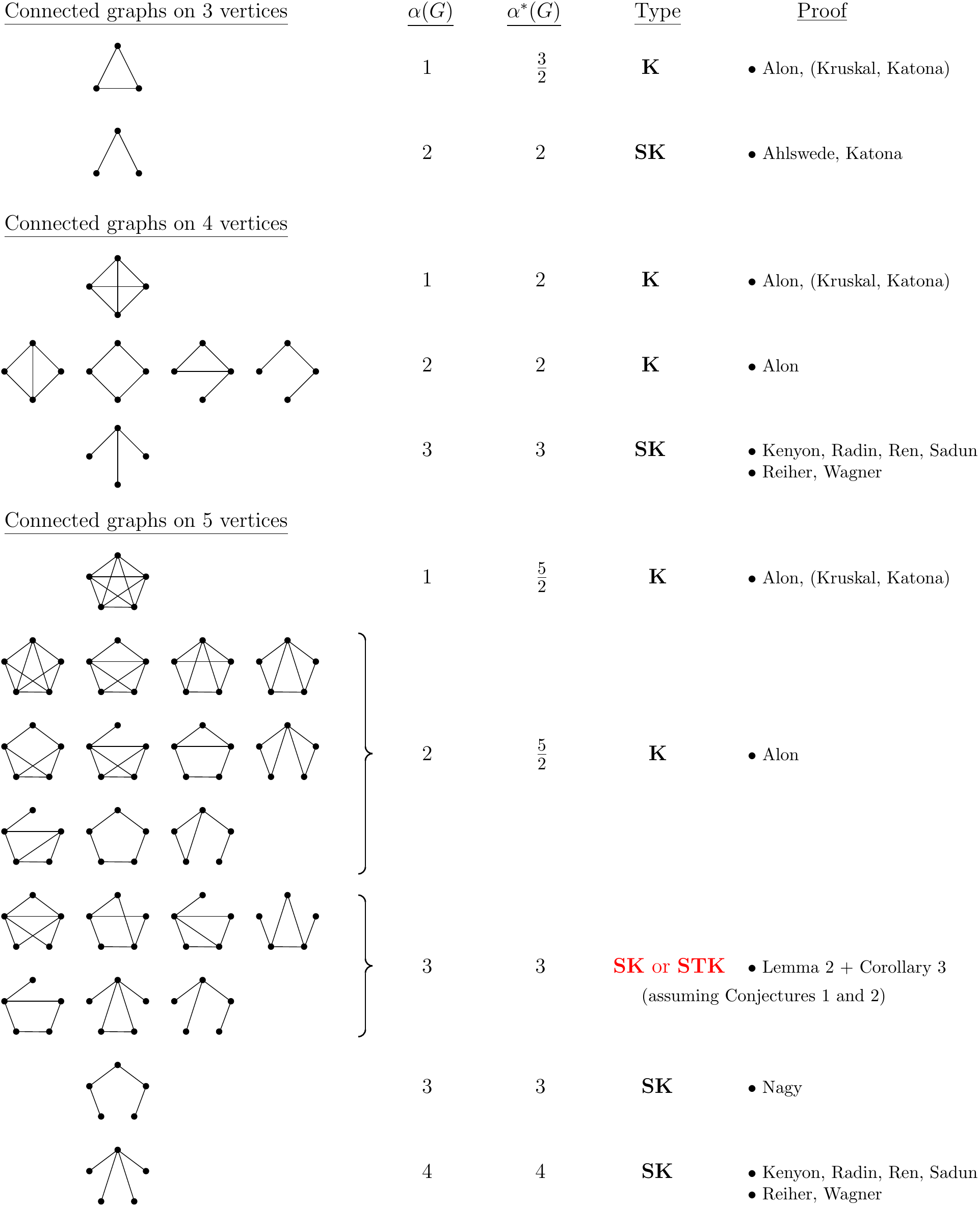}
\caption{All the known types of connected graphs on $5$ or fewer vertices.  The types of all but $7$ of these graphs are known based on the results of the listed authors.  If one assumes that Conjectures 1 and 2 are true, then the remaining $7$ graphs must be of type \textbf{SK} or type \textbf{STK}.}
\label{Fig5}
\end{figure}
\end{document}